\documentclass{article}
\usepackage{amsmath,amssymb,amscd,amsthm,amsfonts}
\usepackage{graphicx,subfigure}

\newtheorem{theorem}{Theorem}
\newtheorem*{theoremp}{Theorem}
\newtheorem{lemma}[theorem]{Lemma}
\newtheorem{corollary}[theorem]{Corollary}
\newtheorem*{problem}{Problem}

\newcommand{\R}{\mathbb{R}}
\newcommand{\C}{\mathcal{C}}
\newcommand{\Z}{\mathbb{Z}}
\newcommand{\lip}{\text{Lip}}

\title{An extension of a theorem of Yao \& Yao}
\author{Edgardo Rold\'an-Pensado \\ e.roldan@math.ucl.ac.uk \\ \and Pablo Sober\'on \\pablo.soberon@ciencias.unam.mx \and ------------------------------------------------------------ \\Mathematics Department \\ University College London \\ London WC1E 6BT
\\ ------------------------------------------------------------}

\begin{document}
\maketitle

\abstract{In this paper we study $N_d(k)$, the smallest positive integer such that any nice measure $\mu$ in $\R^d$ can be partitioned into $N_d(k)$ convex parts of equal measure so that every hyperplane avoids at least $k$ of them. A theorem of Yao and Yao \cite{YY1985} states that $N_d(1) \le 2^d$. Among other results, we obtain the bounds $N_d(2)\le 3\cdot 2^{d-1}$ and $N_d(1)\ge C\cdot 2^{d/2}$ for some constant $C$. We then apply these results to a problem on the separation of points and hyperplanes.

\section{Introduction}

A convex partition of $\R^d$ into $n$ parts is a covering $\mathcal{P}=\{C_1,\dots,C_n\}$ of $\R^d$ consisting of closed convex bodies with pairwise disjoint interiors. We say that a hyperplane $H\subset\R^d$ avoids a set $C$ if it does not intersect its interior.

The classical Yao-Yao theorem \cite{YY1985} states the following.
\begin{theoremp}[Yao and Yao, 1985]
Let $\mu$ be a nice measure in $\R^d$, then there is a convex partition $\mathcal P$ of $\R^d$ into $2^d$ parts of equal $\mu$-measure such that every hyperplane in $\R^d$ avoids at least one element of $\mathcal P$.
\end{theoremp}
In the original proof of this theorem the measure has to have a continuous density function bounded away from $0$. Later, in \cite{Leh2009}, this was weakened to the condition that the measure of every hyperplane be $0$.
In this paper, we ask $\mu$ to satisfy the original conditions, as in \cite{YY1985}.

This theorem gives a partition in which almost every hyperplane intersects exactly $2^d-1$ pieces.
The proof gives a unique partition for each ordered orthonormal basis $(u_1,\dots,u_d)$. We extend this theorem to the case when every hyperplane is required to avoid $2$ pieces.

\begin{theorem}\label{teo:y-y_avoids_2}
Let $\mu$ be a nice measure in $\R^d$, then there is a convex partition $\mathcal{P}$ of $\R^d$ into $3\cdot 2^{d-1}$ parts of equal $\mu$-measure such that every hyperplane in $\R^d$ avoids at least two elements of $\mathcal{P}$.
\end{theorem}

For $d=2$ this follows from a theorem by Buck and Buck \cite{BB1949} which states that $\R^2$ \textit{can be divided into six parts of equal measure by three concurrent lines}. Our method gives in this case a partition by three lines, two of which are parallel. The proof of theorem \ref{teo:y-y_avoids_2} is given in section \ref{seccion:teoremon}. Some details of Yao and Yao's orignal proof are necessary, so a sketch of their methods is given in section \ref{seccion:yaoyao}.

Let $N_d(k)$ be the smallest positive integer such that the following holds:
\textit{For every nice measure $\mu$ on $\R^d$ there exists a partition of $\R^d$ into $N_d(k)$ convex parts of equal measure such that every hyperplane avoids at least $k$ parts}. We call such a partition a $k$-equipartition.
Yao-Yao's theorem and theorem \ref{teo:y-y_avoids_2} are equivalent to the bounds $N_d(1)\leq 2^d$ and $N_d(2)\leq 3\cdot 2^{d-1}$.

There is another number which seems useful. Let $M_d(k,\alpha)$ be the smallest positive integer such that \textit{for every nice measure $\mu$ on $\R^d$ there is a family of $M_d(k,\alpha)$ convex sets such that every hyperplane avoids at least $k$ of them and they all have measure at least $\alpha$}.
Clearly we have
$$N_d(k)\geq M_d\left(k,\frac{1}{N_d(k)}\right).$$

\begin{lemma}\label{lem:inR2}
Let $p \le q$ be non-negative integers, then $M_2(q-p,\frac{p}{2q})\le 2q$.
\end{lemma}

From the proof of this lemma we also obtain a bound for $N_2$.

\begin{corollary}\label{coro:inR2}
$N_2(k)\leq 2k+2$.
\end{corollary}

This strictly improves the bounds that can be obtained by using Yao-Yao's theorem, theorem \ref{teo:y-y_avoids_2} and the formulas of section \ref{seccion:otras} for all $k\leq 15$ except $k=1,2,7,12$ (see Table \ref{table}).

\begin{table}\label{table}
\centering
\begin{tabular}{|c|c||c|c||c|c||c|c|}
\hline
$k$ & $N_2(k)$ & $k$ & $N_2(k)$ & $k$ & $N_2(k)$ & $k$ & $N_2(k)$\\
\hline
$1$ & $4$ & $6$ & $14$ & $11$ & $24$ & $16$ & $32$\\
$2$ & $6$ & $7$ & $16$ & $12$ & $24$ & $17$ & $32$\\
$3$ & $8$ & $8$ & $18$ & $13$ & $28$ & $18$ & $36$\\
$4$ & $10$ & $9$ & $20$ & $14$ & $30$ & $19$ & $36$\\
$5$ & $12$ & $10$ & $22$ & $15$ & $32$ & $20$ & $36$\\
\hline
\end{tabular}
\caption{Bounds for $N_2$.}
\end{table}

For a fixed $d$, we can determine the asymptotic behaviour of $N_d(k)$.
\begin{theorem}\label{teo:limit}
$\displaystyle\lim_{k\rightarrow\infty}\frac{N_d(k)}{k}=1$.
\end{theorem}

This means that the condition of equipartitioning a given measure is not very strong in the sense that, if there are enough parts, it can be done so that every hyperplane avoids almost all of them.
This is the same behaviour as one would expect from a random partition.
These results are proven in section \ref{seccion:otras}.
In section \ref{seccion:ab} they are applied to an apparently simple but annoyingly resistant problem regarding separation of points an hyperplanes.

\begin{problem}
Determine all pairs $(\alpha, \beta) \in \R^2_+$ such that for any finite set $X$ of points in $\R^d$ and any finite set $Y$ of hyperplanes in $\R^d$, there are sets $A\subset X$ and $B\subset Y$ such that:
\begin{itemize}
\item $|A| \ge \alpha |X|$,
\item $|B| \ge \beta |Y|$,
\item no two points in $A$ are separated by a hyperplane in $B$.
\end{itemize}
\end{problem}

Trying to solve this question we obtained a bound for $N_d(1)$ in terms of $h_d(t)$, the measure of a spherical cap of $S^d$ with central angle $t$.  This measure is computed with the usual probability measure in $S^d$. 

\begin{theorem}\label{teo:yyinferior}
Let $A$ be a family of convex sets in $\R^d$ such that the following properties hold
\begin{itemize}
	\item Every set in $A$ has measure at most $\alpha$
	\item Every hyperplane avoid at least one set of $A$
	\item The sum of the measures of the sets in $A$ is not greater than $1$.
\end{itemize}
Then $\alpha$ has to be exponentially small in terms of $d$.  In other words, if $1 \ge \alpha \cdot M_d (1,{\alpha})$, then
$$ \frac{1}{\alpha} \geq \left[ h_d \left(\frac{\pi}{4}\right)\right]^{-1} \approx C\cdot 2^{\frac d 2}$$
for some constant $C>0$.
\end{theorem}

As a consequence we obtain,
\begin{corollary}
$$N_d(1) \ge C \cdot 2^{\frac{d}{2}}.$$
\end{corollary}

For the approximation of the cap measure see \cite{Ball97}, for example. Thus, an exponentially large number of parts is needed. This answers a question by B. Bukh on whether the number of pieces needed is indeed super-polynomial.
This kind of bound can be obtained for $N_d(k)$ for any $k$ in terms of spherical caps of $S^d$.
However, explicit approximations are hard to find.
These results are a consequence of theorems \ref{teo:cotasup} and \ref{teo:puntos} below.

\vspace{6pt}
\noindent{\it Remark.} The Yao-Yao theorem can be generalised in the following way: Given $a_1,a_2,\dots,a_{2^d}>0$ such that $\sum a_i=\mu(\R^d)$, there is a partition of $\R^d$ into $2^d$ convex parts $\{C_1,C_2,\ldots,C_{2^d}\}$ such that $\mu(C_i)=a_i$ for all $i$ and every hyperplane avoids the interior of at least one $C_i$. Theorem \ref{teo:y-y_avoids_2} can also be generalised in the same way.
This is made clear in the next sections.

\section{The $(\alpha,\beta)$ problem}\label{seccion:ab}

The $(\alpha, \beta)$ problem deals with how well behaved points and hyperplanes are with each other in terms of separation.
This problem was told to the authors by I. B\'ar\'any in its first version, mentioned in section $1$.
It has the difficulty that it is not self-dual,
so we work with a second version which does have this property.
Namely,
\begin{problem}[Second version]
Find all pairs $(\alpha, \beta) \in \R^2_+$ such that for any two nice probability measures $\mu_1$, $\mu_2$ in $S^d$ there are sets $A, B \subset S^d$ with $\mu_1 (A) \ge \alpha$, $\mu_2 (B) \ge \beta$ such that either
\begin{eqnarray}
a \cdot b \ge 0 & \text{for all} & a \in A,\, b \in B \label{eq:ab}\\
&\text{ or }&\nonumber\\
a \cdot b \le 0 & \text{for all} & a \in A,\,b \in B.\nonumber
\end{eqnarray}
\end{problem}

Here we may assume that the measures $\mu_1$ and $\mu_2$ are centrally symmetric. If they are not then we can consider the measures given by $\mu_1'(A)=\frac 12(\mu_1(A)+\mu_1(-A))$ and $\mu_2'(A)=\frac 12(\mu_2(A)+\mu_2(-A))$ and obtain the same pairs $(\alpha, \beta)$ for these measures.

Let $\C'_{d}$ be the set of pairs $(\alpha,\beta)$ that satisfy the conditions of the original problem and $\C_d$ be the set of pairs that satisfy the conditions of the second version of the problem.

Note that $(\alpha,\beta)\in\C'_{d}$ if and only if $(\frac\alpha 2,\frac\beta 2)\in\C_{d}$.
This is done simply by embedding $\R^d$ in $\R^{d+1}$ as a hyperplane not containing the origin.
Then every point $a\in\R^d$ corresponds to the pair of points $\{a',-a'\}$ in $S^d$ in the line $Oa$ and every hyperplane $H\in\R^d$ corresponds to the pair of points $\{b',-b'\}$ in $S^d$ such that $b_1b_2$ is orthogonal to every line $Ob$ with $b\in H$.

With this transformation we see that the original problem is essentially equivalent to one similar to the second version but with finite sets of points instead of probability measures.
The change to measures follows from the fact that every nice measure can be approximated by linear combinations of Dirac measures and vice versa.

The sets $A$ and $B$ in this problem can (and will) be taken as the intersection of $S^d$ and a convex cone in $\R^{d+1}$ with apex at the origin.

We shall denote by $M^d$ usual probability measure on $S^d$.
Given $0\leq t<\pi$ and $x\in S^d$, let $\C(d,x,t)$ be the spherical cap of $S^d$ with centre $x$ and central angle $t$, define $h_d(t)$ as its $M^d$-measure.

Given $A\subset S^d$, let $A^\perp$ be the set of points $x\in S^d$ such that there exists $a\in A$ with $a\cdot x=0$. Note that the largest set $B\subset S^d$ that satisfies \eqref{eq:ab} is one of the connected components of the complement of $A^{\perp}$.

With our notation, theorem 2.1 in \cite{FLM1977} states the following.
\begin{theoremp}
Let $A$ be a closed subset of $S^d$ and set $t>0$ so that $M^d(A)=h_d(t)$. Then for every $\varepsilon>0$, $\mu_d(A_\varepsilon)\geq h_d(t+\varepsilon)$,
where $A_\varepsilon$ is the set of points $x\in S^d$ for which there exists $a\in A$ with $\arccos(a\cdot x)<\varepsilon$.
\end{theoremp}

If $\varepsilon=\frac{\pi}{2}$ and $A$ is connected, then $S^d\setminus A_\varepsilon$ is one of the two connected components of $S^d\setminus A^\perp$. Therefore, if $A,B\subset S^d$ satisfy \eqref{eq:ab} and $M^d(A)=h_d(t)$ for some $t>0$, then $M^d(B)\leq h_d(\frac{\pi}{2} - t)$. This is the following theorem.

\begin{theorem}\label{teo:cotasup}
All points in $\C_d$ lie on or below the curve $$\left\{\left(h_d(t),h_d\left(\frac{\pi}{2}-t\right)\right):0\leq t\leq \frac{\pi}{2}\right\}.$$
\end{theorem}

This turns out to be best possible if $d=1$.

\begin{theorem}\label{teo:S1}
$$\C_1=\left\{(\alpha,\beta):\alpha+\beta\leq\frac{1}{2}\right\}.$$
\end{theorem}

The Yao-Yao type partition theorems can be used to find pairs in the $(\alpha, \beta)$ problem.  The following lemma is the main tool for this purpose.

\begin{lemma}\label{lem:parejas}
Let $0\leq \rho \leq 1$. Suppose that for any nice measure $\mu_1$ on $S_d$ there exists a family $F$ of subsets of $S^d$ and a probability measure $\mu_F$ on $F$ such that
\begin{itemize}
\item $\mu_1(A)\geq\alpha$ for all $A\in F$,
\item For every $b\in S^d$, the set $F_b=\{A\in F:A\cap \{b\}^{\perp} \neq\emptyset\}$ is $\mu_F$-measurable and $\mu_F(F_b)\leq\rho$.
\end{itemize}
Then $(\alpha,\frac{1-\rho}{2})\in\C_d$.
\end{lemma}

Using this with $N_d (k)$ and $M_d (k,\alpha)$, we obtain the following.

\begin{theorem}\label{teo:puntos}
For any two positive integers $k$ and $d$,
$$\left(\frac{1}{2N_d(k)},\frac{k}{2N_d(k)}\right)\in\C_d.$$
More generally, if $\alpha >0$,
$$\left(\frac{\alpha}{2},\frac{k}{2M_d(k,\alpha)}\right) \in \C_d.$$
\end{theorem}

With this theorem and theorem \ref{teo:cotasup} we obtain the lower bounds in theorem \ref{teo:yyinferior}. It should be noted that this also implies lower bounds for $N_d(k)$ for any $k$.

Applying the results obtained for $N_d(k)$, we can show the following.

\begin{corollary}\label{coro:parejas}
For any two non-negative integers $k_1$ and $k_2$, not both equal to $0$, we have
$$\frac{1}{2}\left(
\left[\frac{1}{2^d}\right]^{k_1} \left[\frac{1}{3\cdot 2^{d-1}}\right]^{k_2},
1 - \left[1- \frac{1}{2^d} \right]^{k_1} \left[ 1-\frac{1}{3 \cdot 2^{d-2}} \right]^{k_2}
\right) \in \C_d.$$
\end{corollary}
This gives in particular that $\left(\frac{1}{2^{d+1}},\frac{1}{2^{d+1}}\right)\in\C_d$ and $\left(\frac{1}{3\cdot 2^{d}},\frac{1}{3\cdot 2^{d-1}}\right) \in \C_d$. The fact that $\left(\frac{1}{2^{d+1}},\frac{1}{2^{d+1}}\right)\in\C_d$ was obtained earlier in \cite{APP+2005} using a similar method. In Figure \ref{fig:ab} there are plots of these points together with the bound obtained in theorem \ref{teo:cotasup} in dimensions $2$ and $3$.

\begin{figure}
\centering
\mbox{\subfigure{\includegraphics{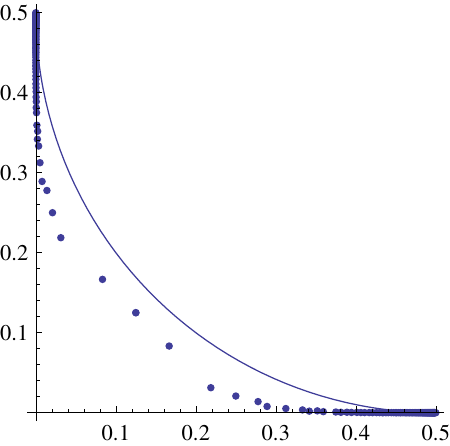}}
\subfigure{\includegraphics{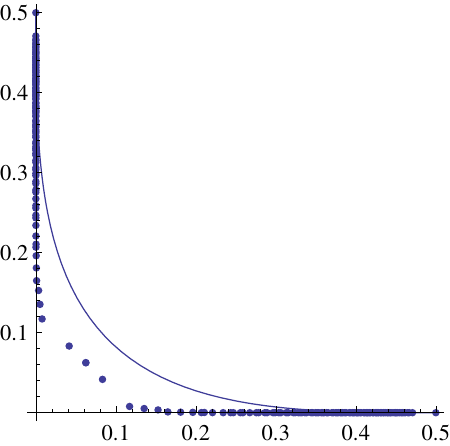}}}
\caption{Bounds for $\C_2$ and $\C_3$.}
\label{fig:ab}
\end{figure}

\begin{corollary}
There are pairs $(\alpha, \beta) \in \C_d$ arbitrarily close to $(0, \frac{1}{2})$.
\end{corollary}

This last corollary comes from the fact that $\lim N_d(k)/k=1$.
However, as these pairs get close to $(0, \frac{1}{2})$ they are significantly worse than what theorem \ref{teo:cotasup} gives.

We can get better bounds if further conditions are imposed on one of the measures. Let $\C_d (\Delta)$ be the set of pairs $(\alpha, \beta)$ such that for any two measures $\mu_1$, $\mu_2$ on $S^d$ such that $\mu_1$ is the integral of a Lipschitz function $f$ with $\lip(f) \le \Delta$, there are sets $A, B$ in $S^d$ satisfying $\mu_1 (A) = \alpha$, $\mu_2 (B) = \beta$ and either $a \cdot b \ge 0$ for all $a \in A$, $b \in B$ or $a \cdot b \le 0$ for all $a \in B$, $b \in B$.

\begin{theorem}\label{teo:derivada}
For every $0<\lambda\leq 1$ and $0<r<\frac{1-\lambda}{\Delta}$ we have,
$$\left(\lambda h_d(r), h_{d-1} \left[\frac{\pi}{2}-2\arcsin\left(\frac{\sin(r)}{\sin(\frac{1-\lambda}{\Delta}-r)}\right)\right] \right)\in \C_d (\Delta).$$
\end{theorem}

If $r$ is close to $0$, then the pairs obtained are close to
$$\left(\lambda h_d (r), h_{d-1}\left[ \frac{\pi}{2} - c_2r\right]\right)$$
for a constant $c_2$ depending on $\lambda$ and $\Delta$.
That is, the difference in dimension with respect to the bounds of theorem $\ref{teo:cotasup}$ is compensated by the constants.
The idea of the proof is to use a small $S^{d-1}$ in $S^d$ and the Lipschitz condition to construct sets as in the proof of lemma \ref{lem:parejas}.
If instead of this we use a hypercube (of dimension $d-1$) in $S^d$, we obtain bounds of the type $(\frac{c_1}{m^{d-1}}, \frac{1}{2} - \frac{c_2 d}{m})$.
These are worse than the ones in theorem \ref{teo:derivada} but are easier to grasp.

\section{Yao-Yao's original proof}\label{seccion:yaoyao}

We quickly go through the original proof by Yao and Yao since its spirit will be followed in the proof of our main theorem. This will also allow us to note some additional properties.

Let $O(d)$ be the space consisting of $d\times d$ matrices $u$ such that $u^Tu=I$, the space $SO(d)\subset O(d)$ consists of matrices with determinant $1$. A matrix $u\in O(d)$ can be expressed as $u=(u_1,\dots,u_d)$ where $u_i$ is the $i$-th row vector of $u$. In this way every $u$ can be identified with an ordered orthonormal base of $\R^d$.

Fix a base $u=(u_1,\dots,u_d)$ of $\R^d$.
If $H$ is a hyperplane orthogonal to $u_1$, define the open half-spaces
\begin{align*}
H^+=&\{x+t u_1:x\in H, t>0\},\\
H^-=&\{x-t u_1:x\in H, t>0\}.
\end{align*}

Let $v$ be a unit vector in $\R^d$ not orthogonal to $u_1$ and let
$p_v:\R^d\rightarrow H$ be the projection such that $p_v(x+t v)=x$ for all $x\in H$ and $t\in\R$.
We can identify $H$ with $\R^{d-1}$ by means of the base $u_2,\dots,u_d$.
There is a natural way to define measures $\mu^+_v$ and $\mu^-_v$ in $H$: For any measurable $S\subset H$, set $\mu^\pm(S)=\mu(p_v^{-1}(S)\cap H^\pm)$.

In \cite{YY1985} a centre $c\in\R^d$ for $\mu$ relative to the base $u_1,\dots,u_d$ is defined as follows:
\begin{itemize}
\item If $d=1$ then $c$ is the point that splits $\R$ into two parts of equal $\mu$-measure.
\item If $d>1$, let $H$ be the hyperplane orthogonal to $u_1$ that splits $\R^d$ into two parts of equal $\mu$-measure. Then $c$ lies on $H$ and there exists a unit vector $v$ (with $u_1\cdot v>0$) such that $c$ is a centre for both $\mu^+_v$ and $\mu^-_v$ relative to $u_2,\dots,u_d$.
\end{itemize}

This induces a partition into $2^d$ parts, if a hyperplane intersects the line through $c$ parallel to $v$ in $H^+$ then it avoids one of the elements of the partition contained in $H^-$ and vice versa.

It is then proved that $c$ is exists and is unique. It is easy to see that the projection vector $v$ is also unique and that $v$ and $c$ vary continuously with $u$.

If we want each element of the partition to have a pre-described value, then the same proof works by changing the choice of $H$ appropriately.

\section{Proof of theorem \ref{teo:y-y_avoids_2}}\label{seccion:teoremon}

Problems involving partitions of measures are topological in nature.
A common method to approach them is to parametrise a subset of the possible partitions by a space $X$ (called phase space), construct a space $Y$ (called target space) of parameters of a partition and relate them by a function $f:X\rightarrow Y$ (called test map).
Ideally, there will be a group acting on both $X$ and $Y$ such that $f$ is equivariant. The existence of the target partition is then reduced to showing that any equivariant function on those spaces always takes some value (e.g. \cite{Mat2003}, \cite{Sob2011}).
We follow this sketch and reduce the problem to showing that some equivariant functions always have a zero.

Given $x=(x_1,\dots,x_n)\in\R^{d_1}\times\dots\times\R^{d_n}$, we define $g_i(x)$ as the result of changing the sign of the $i$-th coordinate of $x$. We always use $g_i$ to denote this function independently of the target space, since it causes no confusion. We denote he $j$-th coordinate of $x_i\in\R^{d_i}$ by $x_i^{(j)}$.

We start with the geometrical part of the proof of theorem \ref{teo:y-y_avoids_2} and continue with the topological part.

\begin{proof}[Proof of theorem \ref{teo:y-y_avoids_2}]
Let $u=(u_1,\dots,u_d)$ be an orthonormal base of $\R^d$, we think of $u_1$ as the upwards direction. If $H$ is a hyperplane orthogonal to $u_1$, define the open half-spaces
\begin{align*}
H^+=&\{x+t u_1:x\in H, t>0\},\\
H^-=&\{x-t u_1:x\in H, t>0\}.
\end{align*}

Let $H_1$ and $H_2$ be the hyperplanes orthogonal to $u_1$ such that the sets $A=H_1^+$, $B=H_1^-\cap H_2^+$ and $C=H_2^-$ have equal $\mu$-measure. Let $\mu_1=\mu|{A\cup B}$ and $\mu_2=\mu|{B\cup C}$.

Yao-Yao's theorem applied to $\mu_1$ gives a unique centre $O_1\in H_1$ and a unique projection vector $v_1$ pointing downwards (i.e. $u_1\cdot v_1<0$).

Let $J_1\subset H_1$ be the $(d-2)$-dimensional flat through $O_1$ orthogonal to $u_2$. Note that the hyperplane $K_1=\{J_1+tv_1:t\in\R\}$ splits $B$ into two parts of equal $\mu_1$-measure.

Define analogously $O_2\in H_2$, $v_2$ pointing upwards, $J_2$ and $K_2$ (see Fig. \ref{fig:y-y_avoids_2}). Since $K_1$ and $K_2$ each divide $B$ into two parts of equal $\mu$-measure, they intersect in a $d-2$-dimensional flat $J\subset B$ parallel to $J_1$ and $J_2$.

\begin{figure}
\centering
\includegraphics{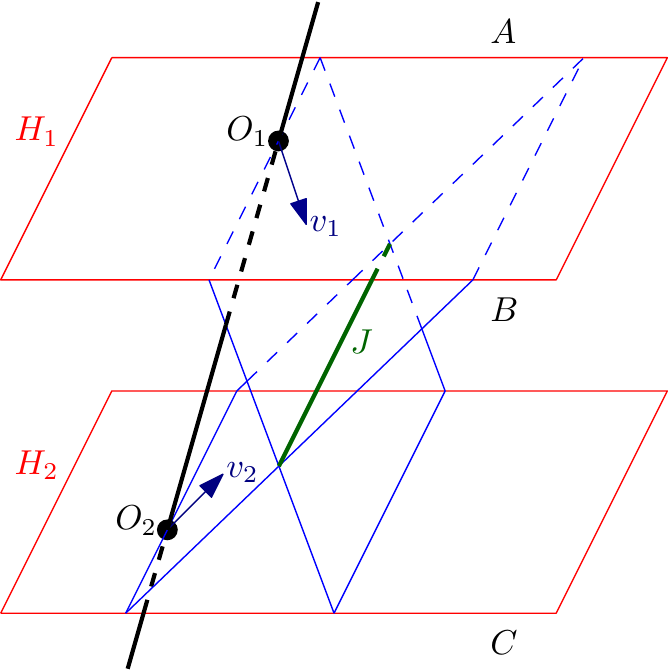}
\caption{The hyperplanes, centres and projection vectors.}
\label{fig:y-y_avoids_2}
\end{figure}
The centres $O_1$ and $O_2$ as well as the vectors $v_1$ and $v_2$ vary continuously with $u$.

Our aim is to find $u$ such that the vectors $v_1$ and $v_2$ are parallel to the line $O_1 O_2$.
Once this is done, let $\mathcal P_1$ and $\mathcal P_2$ be the corresponding Yao-Yao partition to $\mu_1$ and $\mu_2$.

Then we can use the partition $\mathcal P$ consisting of the elements of $\mathcal P_1$ contained in $A$, the elements of $\mathcal P_2$ contained in $C$ and the non-empty elements $K\cap B$ such that $K\in\mathcal P_1$. Every hyperplane avoids at least two elements of $\mathcal P$.  This is because if it hits the line $l$ in section $A$, it misses a section contained in $B$ and one contained in $C$, if it hits $l$ in $B$ it misses a section in $A$ and one in $C$ and if it hits $l$ in $C$ it misses a section in $A$ and one in $B$.

Now we search for the base $u$.

Let $r:\R^d\rightarrow \R^{d-1}$ be the projection such that $r(O_1)=r(O_2)=O$ and $\{r(u_2),\dots,r(u_d)\}$ is the canonical basis.

The affine hyperplane $r(J)$ is orthogonal to $u_2$, so it is of the form $\{v:v\cdot u_2=\lambda\}$ for some $\lambda\in\R$. Let $x\in\R^{d-2}$ and $y\in\R^{d-2}$ be the vectors consisting of the last $d-2$ coordinates of $r(v_1)$ and $r(v_2)$, respectively.

Let $h(u)=(x,y,\lambda)\in\R^{d-2}\times\R^{d-2}\times\R$, note that if $h(u)=0$ for some $u$, then the vectors $v_1$ and $v_2$ are parallel to the line $O_1O_2$ and we are done.

The map $h$ satisfies the following conditions:
\begin{itemize}
\item $h(g_1(u))=(y,x,\lambda)$,
\item $h(g_2(u))=(x,y,-\lambda)$,
\item $h(g_{i+2}(u))=(g_i(x),g_i(y),\lambda)$ for $i=1,\dots,d-3$.
\end{itemize}
Where for any $z$, $g_i(z)$ is the result of changing the sign on the $i$-th coordinate.

Let $f:O(d)\rightarrow\R^{d-1}\times\R^{d-2}\times\dots\times\R^1$ be defined by
$$f(u)=((x+y,\lambda),x-y,0,\dots,0).$$
Finding a zero of $h$ is equivalent to finding a zero of $f$. We will prove something more general, but first we need a definition. Let $v\in\R^{d-1}\times\dots\times\R^1$, then $v^{(j)}=(v_1^{(j)},\dots,v_{d-j}^{(j)})$ and $v^T=(v^{(1)},\dots,v^{(d-1)})$.

\noindent{\bf Claim.} {\it Assume $f:O(d)\rightarrow\R^{d-1}\times\dots\times\R^1$ is a function such that whenever $f(u)=v$, then
\begin{itemize}
\item $f(g_1(u))=g_2(v)$,
\item $f(g_2(u))=g_{d-1}(v^T)^T$,
\item $f(g_{i+2}(u))=g_i(v^T)^T$ for $i=1,\dots,d-3$.
\end{itemize}
Then there exists $u\in O(d)$ such that $f(u)=0$.}

We use a similar proof method to the one B\'ar\'any used to prove the Borsuk-Ulam theorem in \cite{Bar1980}.
This method is thoroughly explained in Chapter 2.2 of \cite{Mat2003} and in \cite{Mus2011}.
We leave the details omitted of the following proof to the interested reader.

Together with composition, we can think of $\{g_1,\dots g_d\}$ as a set of generators of the group $\Z_2^d$. Given the conditions on $f$, there are natural group actions of $\Z_2^d$ on $O(d)$ and $\R^{d-1}\times\dots\times\R^1$ such that $f$ is equivariant. However, the space $O(d)$ is too large for our needs, so instead we consider the restriction $f^*=f|SO(d)$ and the group actions of $\Z_2^{d-1}$ on $SO(d)$ and $\R^{d-1}\times\dots\times\R^1$ obtained by taking the group generated by $\{g_1\circ g_d,\dots,g_{d-1}\circ g_d\}$.

Let $f_0:SO(d)\rightarrow\R^{d-1}\times\dots,\R^1$ be the function given by
$$f_0(u)=(v_1,\dots,v_{d-1}),$$
where
\begin{itemize}
\item $v_1=(u_3^{(1)},\dots,u_d^{(1)},u_2^{(1)})$,
\item $v_2=u_1^{(1)}\cdot(u_3^{(2)},\dots,u_d^{(2)}),$
\item $v_{i+2}=(u_3^{(i+2)},\dots,u_{d-i}^{(i+2)})$ for $i=1,\dots,d-3$.
\end{itemize}
This function is continuous and equivariant. Furthermore, if $f_0(u)=0$ then it is not difficult to see that $u_i$ is the $i$-th element of the canonical basis or its negative. Therefore $f_0$ has exactly $2^{d-1}$ zeros in $SO(d)$.

Let $F:SO(d)\times I\rightarrow\R^{d-1}\times\dots\times\R^1$ be the equivariant homotopy given by $F(u,t)=tf^*(u)+(1-t)f_0(u)$ that takes $f_0$ to $f^*$.

Assume $f^*$ (and therefore $F$) is generic enough, then the set $F^{-1}(0)$ consists of paths and cycles. If $f^*$ has no zeros, then all the paths of $F^{-1}(0)$ have their endpoints at points of the form $(u,0)$ where $u$ is a zero of $f_0$. Therefore there must be a path connecting two such points, but this is impossible due to the group action.

Since almost all continuous functions $f$ are generic, this is valid for all continuous functions $f$.
\end{proof}

\section{Other Proofs}\label{seccion:otras}

\begin{proof}[Proof of theorem \ref{lem:inR2}]
Set a parameter $t\geq 0$, we proceed inductively. Let $\ell_1$ be a oriented halving line (i.e. a line that splits $\R^2$ into two parts of equal $\mu$-measure that has fixed right and left sides). Once that $\ell_k$ has been constructed, let $\ell_{k+1}$ be the oriented halving line such that the sets 
\begin{align*}
A_k=&\{x\in\R^2:x\text{ is right of }\ell_{k+1}\text{ and left of }\ell_k\}\\
A_{q+k}=&\{x\in\R^2:x\text{ is left of }\ell_{k+1}\text{ and right of }\ell_k\}\\
\end{align*}
have $\mu$-measure $\frac{p}{2q}+t$, for $k=1,\dots,q$.
If $t=0$ then the sum of the measures of these sets up to $k=q$ is $p$, but they may not cover $p$ times every point of $\R^2$.
There is a smallest $t$ such that each point of $\R^2$ is covered at least $p$ times.
In this case $\ell_1$ and $\ell_{q+1}$ are equal as sets, so in total we have $q$ lines and the boundary of every $A_k$ is contained in the union of two of them.
If $\ell\subset\R^2$ is a typical line then it intersects each of the lines $\ell_1,\dots,\ell_q$ once, therefore it intersects exactly $p+q$ elements of $F$.
\end{proof}

\begin{proof}[Proof of corollary \ref{coro:inR2}]
In the previous proof put $p=1$ and $q=k+1$, then the resulting sets form a partition.
\end{proof}

\begin{proof}[Proof of theorem \ref{teo:limit}]
Clearly $N_d(k)\geq d+k$, as there is a hyperplane through any given $d$ points. The function $N_d$ also satisfies
\begin{equation}\label{eq:weak}
N_d(k_1+k_2)\leq N_d(k_1)+N_d(k_2).
\end{equation}
To see this, partition $\R^d$ by a hyperplane that divides its measure in proportions $N_d (k_1) : N_d (k_2)$. We can find a $k_1$-equipartition of one side and a $k_2$-equipartition of the other side. We are left with $N_d(k_1)+N_d(k_2)$ parts of equal measure such that every hyperplane avoids $k_1+k_2$ of them.

We also have the asymptotically stronger equation
\begin{equation}\label{eq:strong}
N_d(k_1N_d(k_2)+k_2N_d(k_1)-k_1k_2)\leq N_d(k_1)N_d(k_2).
\end{equation}
This can be shown by finding a $k_1$-equipartition of $\R^d$ and further partition each of its pieces by $k_2$-equipartitions.
We are left with $N_d(k_1)N_d(k_2)$ parts of equal measure such that every hyperplane intersects at most $(N_d(k_1)-k_1)(N_d(k_2)-k_2)$ of them.

Starting with Yao-Yao's theorem and iterating \eqref{eq:strong}, a sequence of partitions can be found such that in the $i$-th step we have $2^{di}$ parts of equal measure and every hyperplane intersects at most $(2^d-1)^i$ of them.
Therefore, a sequence $k_i$ can be found in which $N_d(k_i)/k_i$ tends to $1$. Then \eqref{eq:weak} implies
$$\lim_{k\rightarrow\infty}\frac{N_d(k)}{k}=1.$$
\end{proof}

\begin{proof}[Proof of theorem \ref{teo:S1}]
Take $\alpha,\beta$ with $\alpha+\beta\leq\frac{1}{2}$.
Suppose that for every arc segment $A$ with $\mu_1(A)=\alpha$ we have that $\mu_2(A^\perp)>\alpha$. Note that each of the two components of $A^\perp$ is obtained by rotating $A$ by an angle of $\pm\frac{\pi}{2}$.
This implies that $\mu_1(S^1)<\mu_2(S^1)$, which is a contradiction.
Therefore there exists an arc segment $A$ that satisfies $\mu_2(A^\perp)\leq\alpha$.
If $B$ is any of the two components of $S^1\setminus A^\perp$, then $\mu_2(B)\geq\beta$.
This proves one of the inclusions, theorem \ref{teo:cotasup} gives us the other.
\end{proof}

\begin{proof}[Proof of theorem \ref{lem:parejas}]
Let $\mu_1$ and $\mu_2$ be nice measures. Suppose that we can find $\mu_F$ as above, then by Fubini's theorem
\begin{align*}
\int_F\mu_2(A^\perp)d\mu_F =& \int_F\int_{S^d}\chi(A^\perp)d\mu_2d\mu_F
= \int_{S^d}\int_F\chi(F_b)d\mu_Fd\mu_2\\
=& \int_{S^d}\mu_F(F_b)d\mu_2 \leq \rho.
\end{align*}
Thus, we can find $A_0$ such that $\mu_2 (A_0^\perp)\leq\rho$. This means that there is a set $B$ such that $\mu_2 (B) \ge \frac{1-\rho}{2}$ and the sign of $a\cdot b$ is constant for all $a\in A_0$ and $b\in B$.
\end{proof}

\begin{proof}[Proof of theorem \ref{teo:puntos}]
Given a hyperplane in $\R^{d+1}$ that does not go through $0$, we may identify $\R^d$ with $S^d$ by means of a radial projection.
This gives a measure $\mu^*$ in $\mathbb{R}^d$, so we can find $M_d(k,\alpha)$ convex sets of measure $\alpha$ such that every hyperplane (in $\mathbb{R}^d$) avoids the interior of at least $k$ of them.
If we bring this family of convex sets back to $S^d$, we obtain a new family $F$ of $2M_d(k, \alpha)$ convex sets in $S^d$, each of measure $\frac{\alpha}{2}$.
By choosing $\mu_F$ to be the discrete probability measure on $F$ and applying lemma \ref{lem:parejas}, we are done.
\end{proof}

\begin{proof}[Proof of theorem \ref{teo:derivada}]
We may assume that $\mu_1(S^d)=1$, then there must be a point $x_0\in S^d$ such that $f(x_0)\geq 1$.
Let $\lambda\leq 1$ and $R=\min(\frac{1-\lambda}{\Delta},\frac{\pi}{2})$, from the fact that $\lip(f)\leq\Delta$ it follows that $f\geq\lambda$ on $C(d,x_0,R)$.

Set $r\leq \frac{R}{4}$, we consider the family $F$ of caps $C(d,x,r)$ with $x$ in the boundary of $C(d,x_0,R-r)$. Each of these has measure at least $\lambda h_d(r)$.

Here we need some observations:
\begin{itemize}
\item $C(d,x,r)$ is the intersection of $S^d$ with a ball with centre $x$ and radius $\sin(r)$.
\item $\partial C(d,x_0,R-r)$ is a $(d-1)$-sphere with radius $\sin(R-r)$.
\end{itemize}
Let $\mu_F$ be the usual probability measure on $\partial C(d,x_0,R-r)$.
From this we can see that any hyperplane intersects a portion of $F$ with size at most
$$1-2h_{d-1} \left(\frac{\pi}{2}-2\arcsin\left(\frac{\sin(r)}{\sin(R-r)}\right)\right).$$

We conclude the proof in the same way as lemma \ref{lem:parejas}.
\end{proof}

\bibliographystyle{amsplain}
\bibliography{referencias}
\end{document}